\documentclass[reqno,11pt,centertags]{amsart}
\usepackage{amssymb,amsmath,latexsym,amscd,amsfonts,amsthm,mathrsfs,upref}
\usepackage[utf8]{inputenc}

\newcommand{\bd}{{\mathbb{D}}}

\newcommand{\bz}{{\mathbb{Z}}}
\newcommand{\bc}{{\mathbb{C}}}

\newcommand{\bt}{{\mathbb{T}}}

\newcommand{\cf}{{\mathcal{F}}}
\newcommand{\ct}{{\mathcal{T}}}
\newcommand{\css}{{\mathcal{S}}}

\renewcommand{\l}{\lambda}

\newcommand{\s}{\sigma}

\newcommand{\ep}{\varepsilon}
\renewcommand{\r}{\rho}

\newcommand{\p}{\varphi}
\newcommand{\hp}{{\hat\varphi}}

\renewcommand{\th}{\theta}

\newcommand{\oo}{\Omega}

\renewcommand{\gg}{\Gamma}
\newcommand{\z}{\zeta}

\newcommand{\nt}{\noindent}

\newcommand{\bsl}{\backslash}
\newcommand{\ovl}{\overline}

\newcommand{\lp}{\left(}
\newcommand{\rp}{\right)}

\allowdisplaybreaks
\numberwithin{equation}{section}

\newtheorem{theorem}{Theorem}[section]

\newtheorem{corollary}[theorem]{Corollary}

\newtheorem{proposition}[theorem]{Proposition}
\theoremstyle{definition}
\newtheorem{definition}[theorem]{Definition}

\begin{document}

\title[Discrete spectra of Bergman--Toeplitz operators]
{On discrete spectra of Bergman--Toeplitz operators with harmonic symbols}

\author{L. Golinskii}
\address{B. Verkin Institute for Low Temperature Physics and Engineering, 47 Nauky ave., Kharkiv 61103, Ukraine}
\email{golinskii@ilt.kharkov.ua; leonid.golinskii@gmail.com}

\author{S. Kupin}
\address{IMB, CNRS, Universit\'e de Bordeaux, 351 ave. de la Lib\'eration, 33405 Talence Cedex, France}
\email{skupin@math.u-bordeaux.fr}

\author{J. Leblond}
\address{FACTAS, INRIA Sophia Antipolis-M\'editerran\'ee, 2004 route des Lucioles - BP93, 06902 Sophia Antipolis Cedex, France}
\email{juliette.leblond@inria.fr}

\author{M. Nemaire}
\address{IMB, CNRS, Universit\'e de Bordeaux, 351 ave. de la Lib\'eration, 33405 Talence Cedex, France}
\email{Masimba.Nemaire@math.u-bordeaux.fr}

\address{FACTAS, INRIA Sophia Antipolis-M\'editerran\'ee, 2004 route des Lucioles - BP93, 06902 Sophia Antipolis Cedex, France}
\email{Masimba.Nemaire@inria.fr}

\subjclass[2010]{Primary: 47B35; Secondary: 47A55, 30H20, 3010}
\keywords{Discrete spectrum of a perturbed operator, Lieb--Thirring type inequality, Blaschke type inequality, Hardy--Toeplitz operator, Bergman--Toeplitz operator}

\begin{abstract}
In the present article, we study the discrete spectrum of certain bounded Toeplitz operators with harmonic symbol  
on a Bergman space. Using the methods of classical perturbaton theory and recent results by Borichev--Golinskii--Kupin and 
Favorov--Golinskii, we obtain a quantitative result on the distribution of the discrete spectrum of the operator in the unbounded (outer) 
component of its Fredholm set.
\end{abstract}

\maketitle

\section{Introduction}\label{s0}
Various problems of modern analysis require the study of certain classes of ``model" operators. One of the important families is the class 
of Toeplitz operators and operators related to them. Probably, the most classical objects of this kind are Toeplitz operators on Hardy spaces of 
analytic functions, see Nikolski \cite{nk1} for a complete account on the subject. The applications of operators of this class can be found 
in Nikolski \cite{nk2}. Another ``similar" class is the family of Toeplitz operators on Bergman spaces. Their study started in late 80's of 
the last century, see Zhu \cite{zhu} for a nice overview of the topic.

We proceed with some definitions. For a function $\p\in L^\infty(\bt)$, the \emph{Hardy--Toeplitz operator} $T_\p: H^2(\bt)\to H^2(\bt)$ is defined as
\begin{equation}\label{e02}
T_\p h=P_+(\p h), \quad h\in H^2(\bt),
\end{equation}
where $P_+$ is the well-known Riesz othogonal projection from $L^2(\bt)$ to $H^2(\bt)$, see Garnett \cite{ga}. The function $\p$ is called a 
\emph{symbol} of the operator. For the sake of brevity, we call operator $T_\p$ \eqref{e02} an HT-operator.

The definition of a \emph{Bergman--Toeplitz operator} $T_\psi$ (a BT-operator, for short), is rather similar to the above one. 
Indeed, let $L^2_a(\bd)$ be the closed subspace in $L^2(\bd)$ of analytic on $\bd$ functions. Given $\psi\in L^\infty(\bd)$, set 
\begin{equation}\label{e03}
T_\psi: L^2_a(\bd)\to L^2_a(\bd), \quad T_\psi h=\hat P_+(\psi h), 
\end{equation}
where $\hat P_+$ is the orthogonal projection acting from $L^2(\bd)$ to $L^2_a(\bd)$, see Zhu \cite[Chapter 7]{zhu}.

For a function $\p\in L^\infty(\bt)$, it is sometimes convenient to consider a harmonic function $\hp$ on $\bd$, the harmonic extension of $\p$ to $\bd$, given by
\begin{equation}\label{e021}
\hp(z)=\int_\bt \frac{1-|z|^2}{|t-z|^2}\p(t)\, m(dt), \quad z\in\bd.
\end{equation}
Here, $m$ is the probability Lebesgue measure on $\bt$. Certainly, $\hp\in L^\infty(\bd)$ and $\|\hp\|_\infty\le\|\p\|_\infty$.

Given an HT-operator $T_\p$ (with its symbol defined on $\bt$), we consider its \emph{associated BT-operator} $T_\hp$ with symbol 
$\hp$ given by \eqref{e021}. Notice that even though we use ``similar-looking" notation for an HT-operator $T_\p$ and a BT-operator $T_\hp$, 
the confusion is not possible since the functions $\p$ and $\hp$ are defined on $\bt$ and $\bd$, respectively. The domains of definitions of 
corresponding symbols will be always clear from the context of the discussion.

We shall be concerned with Toeplitz operators $T_\p, T_\hp$ having symbols defined as follows. Set
\begin{equation}\label{e04}
\p(t):=\bar g(t) +f(t), \quad t\in\bt, \quad f,g\in H^\infty(\bt).
\end{equation}
Clearly, we have
\begin{equation}\label{e041}
\psi(z):=\hp(z)=\bar g(z)+f(z), \quad z\in \bd.
\end{equation}
It is plain that $\hp$ has the non-tangential boundary values on the unit circle
$$
\hp(t)=\lim_{r\to 1-}\hp(rt), \quad \mathrm{for \ a.\,e.} \ t\in\bt,
$$
and these boundary values coincide with $\p$ a.e. on $\bt$. 

Despite the similarity of definitions \eqref{e02}, \eqref{e03}, the BT- operators exhibit considerably reacher spectral behavior as compared to HT-operators. 
For instance, the essential spectrum 
of HT-operator $T_\p$ is connected, see Widom \cite{wi}, while, in general,  the essential spectrum of BT-operator $T_\psi$ is not. 
There are non-trivial compact BT-operators with quite simple (even radial) symbols \cite[Sections 7.2, 7.3]{zhu}, while a compact HT-operator 
is necessarily zero \cite[Part B, Chapter 4]{nk1}.

Sundberg--Zheng \cite{suzh} showed that there are BT-operators with harmonic symbols having isolated eigenvalues in their spectrum. In subsequent papers, 
Zhao--Zheng \cite{zzh3}, Guan--Zhao \cite{gzh1} and Guo--Zhao--Zheng \cite{gzh2} presented a class of BT-operators with harmonic symbols posessing ``rather big" 
\emph{discrete spectrum}, that is, the set of isolated eigenvalues of finite algebraic multiplicity. 

So, in contrast to HT-operators, the notion of the discrete spectrum of a BT-operator with harmonic symbol makes sense. The study of the pro\-perties 
of the discrete spectrum for BT-operators with symbols \eqref{e041} is the core of the present paper. Unlike the articles \cite{gzh1,gzh2,zzh3}, 
our results are essentially based on the perturbation techniques from operator theory and function-theoretic results of Borichev--Golinskii--Kupin \cite{bgk1, bgk2} 
and Favorov--Golinskii~\cite{fago15}.

We also need the definition of the Sobolev space $W^{1,2}(\bt)$ of absolutely continuous functions on the unit circle $\bt$ with derivative in $L^2$:
$$
W^{1,2}(\bt):=\{h: \bt\to \bc, \ h\in AC, \ h'\in L^2(\bt)\}.
$$

For further purposes, we would like to introduce two closely related cha\-ra\-cte\-ristics of compact sets on the complex plane $\bc$. 
The following definitions are borrowed from Perkal \cite[Section 2]{per} and Peller \cite[Section 4]{pel}, respectively.

\begin{definition}\label{d1}
Let $r>0$. A closed set $E\subset\bc$ is called {\sl $r$-convex} if
$$ 
\bc\bsl E=\bigcup\Bigl\{B(x,r): \ B(x,r)\subset \bc\bsl E\Bigr\}, 
$$
that is, the complement to $E$ can be covered with open disks of a fixed radius $r>0$, which lie in that complement.
\end{definition}

\begin{definition}\label{d2}
A compact set $K\subset\bc$ is called {\sl circularly convex}, if there is $r>0$ such that for each $\l\in\bc\bsl K$ with ${\rm dist} (\l,K)<r$
there are points $\mu\in\partial K$ and $\nu\in \bc\bsl K$ so that
$$ 
|\mu-\nu|=r, \quad \l\in (\mu,\nu], \quad \{\z:\ |\nu-\z|<r\}\subset \bc\bsl K. 
$$
\end{definition}

\smallskip
For example, if $K$ is a convex set, or the boundary $\partial K$ is of $C^2$-class (without intersections and cusps), then $K$ is a circularly convex set. 

Note that the later definition is a bit more stringent than the former one. When $K$ is a (closed) Jordan curve (a rectifiable continuous curve 
with no self-intersections), one can also see that the above definitions are equivalent. 

A short reminder on standard notions and notations from operator theory is given in Subsection \ref{s11} below. For instance, see \eqref{e7} for 
the notion of the unbounded (outer) open component of the Fredholm domain $\cf_0(T)$.

The main result of this note is the following theorem.

\begin{theorem}\label{t1}
Let $T_\p$ be an HT-operator with the symbol $\p$ \eqref{e04} from $W^{1,2}(\bt)$, $\hp$ its harmonic extension \eqref{e021}, and $T_\hp$ be the BT-operator
associated to $T_\p$. Assume that the spectrum $\s(T_\p)$ is a circularly convex set. Then, for each $\ep>0$
\begin{equation}\label{btdisc0}
\sum_{\l\in\s_d(T_\hp)\cap\cf_0(T_\p)} {\rm dist}^{3+\ep}\,(\l,\s(T_\p))\le C(\p,\ep)\,\|\p'\|_2^2.
\end{equation}
\end{theorem}

\begin{corollary}\label{c1}
Let $q$ and $p$ be algebraic polynomials, $\p=\ovl{q}+p$ be a harmonic polynomial, and assume that the image $\p(\bt)$ is a Jordan curve
without cusps. Then $\eqref{btdisc0}$ holds for the discrete spectrum of BT-operator~$T_\hp$.
\end{corollary}

\section{Some preliminaries}\label{s1}
\subsection{Generalities from operator theory}\label{s11}
In this section, we recall some well-known notions of the classical operator theory, see Kato \cite[Section IV.5]{ka1}. 

Let $T$ be a bounded linear operator on a (separable) Hilbert space $H$. As usual, the resolvent set of $T$ is
\begin{equation}\label{e5}
\r(T):=\{\l\in\bc: \ (T-\l):H\to H \ \mathrm{is \ bijective} \ \}.\\
\end{equation}
It follows that $(T-\l)^{-1}$ is bounded for $\l\in \r(T)$. The spectrum of $T$ is defined as 
\begin{equation}\label{e6}
\s(T)=\bc\bsl\r(T).
\end{equation}

Furthermore, we say that a bounded operator $T$ is \emph{Fredholm}, if its kernel and co-kernel are of finite dimension. The essential spectrum 
of $T$ is defined as
$$
\s_{ess}(T)=\{\l\in\bc: (T-\l) \ \mathrm{is \ not \ Fredholm}\}.
$$
One can see that $\s_{ess}(T)$ is a closed subset of $\s(T)$. One considers also the Fredholm domain of $T$, $\cf(T)=\bc\bsl\s_{ess}(T)$. 
Clearly, $\r(T)\subset \cf(T)$. We represent $\cf(T)$ as 
\begin{equation}\label{e7}
\cf(T)=\bigcup^\infty_{j=0} \cf_j(T),
\end{equation}
where $\cf_j(T)$ are disjoint (open) connected components of the set. We agree that 
$\cf_0(T)$ stays for the unbounded connected component of $\cf(T)$.

The discrete spectrum $\s_d(T)$ of $T$ is the set of all isolated eigenvalues of $T$ of finite algebraic multiplicity. For convenience, we put
\begin{equation}\label{e01}
\s_0(T):=\s_d(T)\cap \cf_0(T)\subset \s_d(T).
\end{equation}

Let $A_0, A$  be bounded operators on a Hilbert space such that $A-A_0$ is compact. The operators $A$ and $A_0$ are called compact perturbations 
of each other. The celebrated Weyl's theorem states that 
\begin{equation}\label{e8}
\s_{ess}(A)=\s_{ess}(A_0),
\end{equation}
see Kato \cite[Section IV.5.6]{ka1}. 

We shall be interested in the situation when $\s_0(A)$ is at most countable set, $\s_0(A)=\{\l_j\}_{j\ge 1}$ and it accumulates to the essential 
spectrum $\s_{ess}(A_0)$ only. 

\subsection{Reminder on Hilbert-Schmidt operators}\label{s12}
In this subsection, we recall briefly the notion of a Hilbert-Schmidt operator and its simplest pro\-perties, see Birman-Solomyak \cite[Section 11.3]{biso}. 

Let $A$ be a compact operator. The sequence of \emph{singular values} $\{s_j(A)\}_{j\ge 1}$ is defined as
$$
s_j(A)=\l_j(A^*A)^{1/2}, \qquad s_j(A)\ge0,
$$
where $\l_j(A^*A)$ are eigenvalues of the compact operator $A^*A$. Without loss of generality one can suppose that $\{s_j(A)\}_{j\ge 1}$ 
forms a decreasing sequence, and, moreover
$$
\lim_{j\to+\infty} s_j(A)=0.
$$
One says that $A\in \css_2$, the Hilbert-Schmidt class of compact operators, iff
$$
\|A\|^2_{\css_2}:=\sum_{j\ge 1} s_j(A)^2<\infty.
$$
Equivalently, $A\in \css_2$ if and only if $\{s_j(A)\}_{j\ge 1}\in \ell^2$. Alternatively, the $\css_2$-norm of the operator can be computed as
$$
\|A\|^2_{\css_2}=\sum_{j,k\ge 1}|(Ae_j,e_k)|^2,
$$
where $\{e_j\}_{j\ge 1}$ is an arbitrary orthonormal basis in the given Hilbert space.

\subsection{On the discrete spectrum of a perturbed operator: a result of Favorov--Golinskii}\label{s13}
Some useful quantitative bounds for the rate of convergence of the discrete spectrum of a perturbed operator are given in Favorov--Golinskii \cite[Section~5]{fago15}.
A special case of \cite[Theorem 5.1]{fago15} (cf. a remark right after its proof and formula (5.8)) looks as follows.

\begin{theorem}[{\cite{fago15}}]
\label{t2}
Let $A_0$ be a bounded linear operator on a Hilbert space, which satisfies the conditions:
\begin{enumerate}
\item The spectrum $\s(A_0)$ is an $r$-convex set.
\item The resolvent $R(z,A_0)=(A_0-z)^{-1}$ is subject to the bound
\begin{equation}\label{boures}
\|R(z,A_0)\|\le\frac{C(A_0)}{{\rm dist}^p(z,\s(A_0))}\,, \qquad p>0, \quad z\in\cf_0(A_0).
\end{equation}
\end{enumerate}
Let $B$ be a Hilbert--Schmidt operator, and $A=A_0+B$. Then for each $\ep>0$
\begin{equation}\label{litir}
\sum_{\l\in\s_d(A)\cap\cf_0(A_0)} {\rm dist}^{2p+1+\ep}\,(\l,\s(A_0))\le C(\s(A_0),p,\ep)\,\|B\|_2^2.
\end{equation}
\end{theorem}

If $\s_{ess}(A_0)$ does not split the plane, and \eqref{boures} holds for all 
$\l\in\bc\bsl\s(A_0)$, then \eqref{litir} is true for the whole discrete spectrum $\s_d(A)$. 
For the class of (non-selfadjoint) HT-operators $A_0=T_\p$, $\p$ in \eqref{e04}, the essential spectrum, in general, splits the plane.

\section{Proof of the main result}\label{s2}
Let $\p$ be as in \eqref{e04}. Consider the HT-operator $T_\p$ \eqref{e02} and the associated BT-operator $T_\hp$ \eqref{e03}. 
The technical way to compare these operators is to look at their matrices in appropriately chosen bases. Namely, define
\begin{equation}\label{e05}
e_{H,n}(t)=t^n, \quad e_{B,n}(z)=\sqrt{n+1}\, z^n, \quad n\ge 0.
\end{equation}
It is plain that the systems $\{e_{H,n}\}_{n\ge 0}$ and $\{e_{B,n}\}_{n\ge 0}$ are the orthonormal bases in $H^2(\bt)$ and $L^2_a(\bd)$, respectively. Set
\begin{equation}\label{e07}
\ct_\p=[(T_\p e_{H,i}, e_{H,j})_{H^2(\bt)}]_{i,j\ge 0},\quad 
\ct_\hp=[(T_\hp e_{B,i}, e_{B,j})_{L^2_a(\bt)}]_{i,j\ge 0}.
\end{equation}
The operators $\ct_\p$ and $\ct_\hp$ are unitarily equivalent to original operators $T_\p$ and $T_\hp$, and they both act on $\ell^2(\bz_+)$. 
So, one can argue on the operators $\ct_\p$ and $\ct_\hp$ being ``close" in a certain sense. 

Rewriting relation \eqref{e04} in more detailed form, we have
\begin{equation}\label{harmsym}
\begin{split}
\hp(z) &=\ovl{g(z)}+f(z), \\ 
f(z) &=\sum_{k=0}^\infty f_kz^k\in H^\infty(\bt), \qquad g(z)=\sum_{k=0}^\infty g_kz^k\in H^\infty(\bt), \\
\p(e^{i\th}) &=\hp(e^{i\th})=\sum_{j\in\bz} b_je^{ij\th}\,, \quad b_j=
\begin{cases} f_j,\ & j\ge0, \\
      \ovl{g_{-j}}, \ & j<0.
\end{cases}
\end{split}
\end{equation}

\begin{proposition}\label{hilsch}
Assume that the symbol $\p$ \eqref{e04} belongs to $W^{1,2}(\bt)$. Then $\ct_\hp-\ct_\p$ is a Hilbert-Schmidt operator and
\begin{equation}\label{difnorm}
\|\ct_\hp-\ct_\p\|_{\css_2}^2\le \frac{\pi^2}{24}\,\|\p'\|_2^2.
\end{equation}
\end{proposition}
\begin{proof}
The matrix representation of $\ct_\p$ is obvious: $\ct_\p=[b_{i-j}]_{i,j\ge0}$. So let us compute the matrix $\ct_\hp=[\tau_{i,j}]_{i,j\ge0}$ in the orthonormal basis $\{e_{B,n}\}_{n\ge 0}$ \eqref{e05}. 
For $l,k\ge0$
\begin{equation*}
\begin{split}
\tau_{k,k+l} &=(T_\hp e_{B, k+l}, e_{B, k})_{L^2_a(\bd)}=\frac1{\pi}\int_\bd \hp(z)e_{B, k+l}(z)\,\ovl{e_{B,k}(z)}\,dxdy \\
&= \frac{\sqrt{(k+l+1)(k+1)}}{\pi}\,\int_\bd \hp(z)z^{k+l}\,\ovl{z^k}\,dxdy,
\end{split}
\end{equation*}
or, in polar coordinates,
\begin{equation*}
\tau_{k,k+l}=\frac{\sqrt{(k+l+1)(k+1)}}{\pi}\,\int_0^1\int_0^{2\pi} \sum_{n\in\bz} b_nr^{|n|+2k+l+1}\,e^{i(n+l)\th}\,drd\th.
\end{equation*}
Finally,
\begin{equation*}
\tau_{k,k+l}=\sqrt{\frac{k+1}{k+l+1}}\,b_{-l}, \qquad k,l\ge0.
\end{equation*}
The same formula holds for $\tau_{k+l,k}$, \ $k,l\ge0$, and so
\begin{equation}
\tau_{i,j}=\sqrt{\frac{\min(i,j)+1}{\max(i,j)+1}}\,b_{i-j}, \qquad i,j\ge0.
\end{equation}

Let us estimate the Hilbert-Schmidt norm
\begin{equation*}
\|\ct_\hp-\ct_\p\|_{\css_2}^2=\sum_{i,j\ge0}|\tau_{i,j}-b_{i-j}|^2=\sum_{i,j\ge0} |b_{i-j}|^2\,\left|1-\sqrt{\frac{\min(i,j)+1}{\max(i,j)+1}}\right|^2.
\end{equation*}
As
$$ \sum_{i,j\ge0} a_{ij}=\sum_{l=0}^\infty \sum_{k=0}^\infty a_{k,k+l}+\sum_{l=1}^\infty \sum_{k=0}^\infty a_{k+l,k}, $$
we have
\begin{equation*}
\begin{split}
  \|\ct_\hp-\ct_\p\|_{\css_2}^2 &=\sum_{l=0}^\infty l^2\left[|b_{l}|^2 + |b_{-l}|^2 \right] \, \sum_{k=0}^\infty \frac1{(k+l+1)\Bigl(\sqrt{k+l+1}+\sqrt{k+1}\Bigr)^2}\\
 &=\sum_{l\in \bz} |l \, b_{l}|^2 \, \sum_{k=0}^\infty \frac1{(k+|l|+1)\Bigl(\sqrt{k+|l|+1}+\sqrt{k+1}\Bigr)^2}\\
  & \le \lp\sum_{l \in \bz} |lb_l|^2\rp\, \lp\sum_{k=0}^\infty \frac1{(k+1) \, (2 \, \sqrt{k+1})^2}\rp\\
  & =\sum_{l \in \bz} |lb_l|^2\ \sum_{k=0}^\infty \frac1{ 4 \, (k+1)^2}
=\frac{\pi^2}{4\cdot 6}\,\|\p'\|^2_2=\frac{\pi^2}{24}\,\|\p'\|_2^2.
\end{split}
\end{equation*}
as claimed.
\end{proof}

Weyl's theorem concerning spectra of compact perturbations, mentioned above, leads to the following
\begin{corollary}\label{weyl}
Let the symbol $\p$ satisfy hypothesis of the above Proposition. Then the essential spectrum of BT-operator $T_\hp$ is
\begin{equation*}
\s_{ess}(T_\hp)=\s_{ess}(T_\p)=\p(\bt)=:\gg,
\end{equation*}
and the discrete spectrum $\s_d(T_\hp)$ is at most countable set of eigenvalues of finite algebraic multiplicity with all its accumulation points on $\gg$.
\end{corollary}

\smallskip
We go on with the proof the quantitative version of the above Corollary. 

\medskip
\nt\emph{Proof of Theorem \ref{t1}.} \ \ 
Let $A_0=\ct_\p$, $A=\ct_\hp$, see \eqref{e07}. We only have to ensure that the conditions of Theorem \ref{t2} are met.

It is clear that $\p\in W^{1,2}(\bt)$ implies that $\p\in W$, the Wiener algebra of absolutely convergent Fourier series.
By \cite[Theorem~4]{pel}, the resolvent $(A_0-z)^{-1}$ admits the linear growth, that is, \eqref{boures} holds with $p=1$.
Next, by Proposition \ref{hilsch}, the difference $A-A_0$ is the Hilbert-Schmidt operator with the norm bound \eqref{difnorm}. The proof is complete.
\hfill $\Box$

\medskip
\nt{\bf Acknowledgements. } \  SK, JL, MN are partially supported by
the project ANR-18-CE40-0035.

%
%


\begin{thebibliography}{99}

\bibitem{biso} 
Birman, M.; Solomjak, M. Spectral theory of selfadjoint operators in Hilbert space.  Mathematics and its Applications (Soviet Series). D. Reidel Publishing Co., Dordrecht, 1987.

\bibitem{bgk1}
Borichev, A.; Golinskii, L.; Kupin, S. A Blaschke-type condition and its application to complex Jacobi matrices. Bull. Lond. Math. Soc. 41 (2009), no. 1, 117--123. 

\bibitem{bgk2}
Borichev, A.; Golinskii, L.; Kupin, S. On zeros of analytic functions satisfying non-radial growth conditions. Rev. Mat. Iberoam. 34 (2018), no. 3, 1153--1176.

\bibitem{fago15}
Favorov, S.; Golinskii, L. Blaschke-type conditions on unbounded domains, generalized convexity, and applications in perturbation theory, Rev. Mat. Iberoam. {\bf 31} (1) (2015), 1--32.

\bibitem{ga}
Garnett, J. Bounded analytic functions. Revised first edition. Graduate Texts in Mathematics, 236. Springer, New York, 2007.

\bibitem{gzh1}
Guan, N.; Zhao, X.  Invertibility of Bergman Toeplitz operators with harmonic polynomial symbols. Sci. China Math. 63 (2020), no. 5, 965–978.

\bibitem{gzh2}
Guo, K.; Zhao, X.; Zheng, D.
The spectral picture of Bergman Toeplitz operators with harmonic polynomial symbols, submitted, \emph{https://arxiv.org/abs/2007.07532}.

\bibitem{ka1}
Kato, T. Perturbation theory for linear operators. Reprint of the 1980 edition. Classics in Mathematics. Springer-Verlag, Berlin, 1995.

\bibitem{nk1}
Nikolski, N. Operators, functions, and systems: an easy reading, I. Hardy, Hankel, and Toeplitz. Mathematical Surveys and Monographs, Vol. 92. American Mathematical Society, Providence, RI, 2002.

\bibitem{nk2}
Nikolski, N. Operators, functions, and systems: an easy reading, II. Model operators and systems. Mathematical Surveys and Monographs, Vol. 93. American Mathematical Society, Providence, RI, 2002.

\bibitem{pel}
Peller V., Spectrum, similarity, and invariant subspaces of Toeplitz operators, Mathematics of the USSR-Izvestiya, {\bf 29} (1) 
(1987), 133--144.

\bibitem{per}
Perkal J., Sur les ensembles $\ep$-convexes, Colloq. Mat., {\bf 4} (1956), 1--10.

\bibitem{suzh}
Sundberg, C.; Zheng, D.
The spectrum and essential spectrum of Toeplitz operators with harmonic symbols.
Indiana Univ. Math. J. 59 (2010), no. 1, 385--394.

\bibitem{wi}
Widom, H. On the spectrum of a Toeplitz operator. Pacific J. Math. 14 (1964), 365--375.

\bibitem{zzh3}
Zhao, X.; Zheng, D. The spectrum of Bergman Toeplitz operators with some harmonic symbols.  Sci. China Math. 59 (2016), no. 4, 731–740.

\bibitem{zhu}
Zhu, K. Operator theory in function spaces. Second edition. Mathematical Surveys and Monographs, 138. American Mathematical Society, Providence, RI, 2007.

\end{thebibliography}
\end{document}